\documentclass[a4,12pt]{article}%

\usepackage{graphicx}
\usepackage{graphics}

\usepackage{amsmath}
\usepackage{amsfonts}
\usepackage{amssymb}
\usepackage{enumerate}
\usepackage{xypic}
\newtheorem{theorem}{Theorem}[section]

\newtheorem{definition}[theorem]{Definition}

\newtheorem{lemma}[theorem]{Lemma}

\newenvironment{proof}[1][Proof]{\textbf{#1.} }{\ \rule{0.5em}{0.5em}}
\newcommand{\dom}{{\rm domain}}

\newcommand{\E}{{\rm \bf E}}

\newcommand{\image}{{\rm image}}
\newcommand{\codim}{{\rm codim}}

\newcommand{\rmd}{{\mathrm d}}
\newcommand{\dist}{{\rm dist}}

\newcommand{\dN}{{\bf N}}

\newcommand{\dR}{{\bf R}}

\newcommand{\ep}{\varepsilon}

\newcounter{figurecounter}
\begin{document}

\title{Sunspot Equilibrium in Positive Recursive General Quitting Games%
\thanks{The authors thank Hari Govindan, Ehud Lehrer, and John Levy for useful discussions.
E. Solan acknowledges the support of the Israel Science Foundation, grant \#217/17.}}

\author{Eilon Solan and Omri N.~Solan%
\thanks{The School of Mathematical Sciences, Tel Aviv
University, Tel Aviv 6997800, Israel. e-mail: eilons@post.tau.ac.il, omrisola@post.tau.ac.il.}}

\maketitle

\begin{abstract}
We prove that positive recursive general quitting games, which are quitting games in which each player may have more than one continue action, admit a
sunspot $\ep$-equilibrium, for every $\ep > 0$.
To this end we show that the equilibrium set of strategic-form games can be uniformly approximated by a smooth manifold,
and develop a new fixed-point theorem for smooth manifolds.
\end{abstract}


\section{Introduction}

One of the central open questions in game theory to date is whether every multiplayer stochastic games admits a uniform equilibrium payoff.
Mertens and Neyman (1981) proved that two-player zero-sum stochastic games admit a uniform value,
Vieille (2000a, 2000b) proved that two-player nonzero-sum stochastic games admit a uniform equilibrium payoff,
and Solan (1999) proved that three-player absorbing games admit a uniform equilibrium payoff.
Solan and Vieille (2001) presented the family of quitting games, and showed that a certain class of multiplayer quitting games admit a uniform equilibrium payoff.
Further results regarding the existence of uniform equilibrium in quitting games were proven by Simon (2012) and Solan and Solan (2019).

While the existence of a uniform equilibrium payoff in general stochastic games is still an open problem,
the existence of an extensive-form correlated equilibrium payoff in multiplayer stochastic game
was proven by Solan and Vieille (2002).
Recall that an extensive-form correlated equilibrium payoff is a uniform equilibrium payoff in an extended game,
which includes a correlation device that sends at every stage a private signal to each player,
where the signal can depend on past signals sent to all players.
Solan and Vohra (2001, 2002) proved that every absorbing game admits a normal-form correlated equilibrium payoff,
which is a uniform equilibrium payoff in an extended game that includes a correlation device
that sends one private signal to each player at the outset of the game.

Recently Solan and Solan (2019) proved that every quitting game admits a sunspot equilibrium payoff,
which is an equilibrium payoff in an extended game that includes a correlation device
that sends at every stage a public signal that is  uniformly distributed on $[0,1]$ and independent of past signals and play.

In this paper we extend the result of Solan and Solan (2019) to a more general class of absorbing games,
namely, the class of \emph{positive recursive general quitting games}.
These are quitting games in which
(a) each player has a single quitting action and possibly several continue actions,
(b) the nonabsorbing payoff is 0,
and (c) the absorbing payoffs are nonnegative.

In addition to proving that a sunspot equilibrium payoff exists in the class of positive recursive general quitting games,
the paper has several contributions, which are needed in the proof of the main result.
\begin{itemize}
\item   We show that the equilibrium set can be uniformly approximated by smooth manifolds,
a property that allows us to use topological results that require manifolds to be smooth.
\item   We develop a new fixed point result for smooth manifolds.
\item   We develop a new technique for studying multiplayer absorbing games, which reduces an absorbing game into a collection of quitting games.
\item   As noted by Solan, Solan, and Solan (2018), our results imply that if at least two players
have at least two continue actions, then the positive recursive general quitting game admits a uniform equilibrium payoff.
\end{itemize}

The paper is organized as follows.
The model and the main game theoretic result are described in Section~\ref{section:model}.
The proof for the case in which one player has two continue actions and all other players
have one continue action appears in Section~\ref{section:sketch}.
Section~\ref{section:topology} presents the results in topology that we need in the main proofs,
and Section~\ref{section:equilibrium} shows that the equilibrium set can be uniformly approximated by smooth manifolds.
In Section~\ref{section:proof} we provide the proof of the main result.
Section~\ref{section:extensions} discusses extensions of our main game theoretic result to other classes of absorbing games.

\section{The Model and Main Results}
\label{section:model}

\subsection{General Quitting Games}

\begin{definition}
A \emph{general quitting game} is a vector $\Gamma = (I,(A_i^c)_{i \in I},u)$ where
\begin{itemize}
\item   $I$ is a finite set of \emph{players}.
\item   $A_i^c$ is a finite nonempty set of \emph{continue actions},
for each player $i \in I$.
The set of all \emph{actions} of player~$i$ is $A_i := A_i^c \cup \{Q_i\}$, where $Q_i$ is interpreted as a \emph{quitting action}.
The set of all \emph{action profiles} is $A := \times_{i \in I} A_i$
and the set of all \emph{absorbing action profiles} is $A_* := A \setminus \left(\times_{i \in I} A_i^c \right)$.
\item   $u : A \to [-1,1]^I$ is a \emph{payoff function}.
\end{itemize}
\end{definition}

The game proceeds as follows.
At every stage $t \in \dN$, each player $i \in I$ chooses an action $a_i^t \in A_i$.
Let $a^t := (a_i^t)_{i \in I}$ be the action profile chosen at stage $t$.
We denote by $t_*$ the first stage in which a quitting action is played;
that is,
\[ t_* := \min\{ t \in \dN \colon a^t \in A_*\}, \]
with the convention that the minimum of an empty set is $\infty$.

A \emph{mixed action} of player $i$ is an element of $\Delta(A_i)$.
Each action $a_i \in A_i$ is identified with the mixed action that assigns probability 1 to $a_i$.
A \emph{(behavior) strategy} of player~$i$ is a function $\sigma_i \colon \left(\cup_{t=0}^\infty A^t\right) \to \Delta(A_i)$.
A \emph{strategy profile} is a vector of strategies $\sigma = (\sigma_i)_{i \in I}$, one for each player.
We identify each mixed action profile $x \in X := \times_{i \in I} \Delta(A_i)$ with the stationary strategy profile that plays $x$ at every stage.
Every strategy profile $\sigma$ induces a probability distribution over the set of plays
$A^\infty$.
Denote by $\E_\sigma$ the corresponding expectation operator.
The (undiscounted) \emph{payoff} under strategy profile $\sigma$ is
\[ \gamma(\sigma) := \E_\sigma\left[\limsup_{T \to \infty} \frac{1}{T}\sum_{t=1}^T u(a^{\min\{t,t_*\}})\right]. \]
Thus, the play is effectively terminated at stage $t_*$.

\begin{definition}
Let $(I,(A_i^c)_{i \in I},u)$ be a general quitting game and let $\ep \geq 0$.
A strategy profile $\sigma = (\sigma_i)_{i \in I}$ is an \emph{$\ep$-equilibrium}
if for every player $i \in I$ and every strategy $\sigma'_i$ of player~$i$,
\[ \gamma_i(\sigma) \geq \gamma_i(\sigma'_i,\sigma_{-i}) - \ep. \]
\end{definition}

The equilibrium concept that we study in this paper is \emph{undiscounted equilibrium}.
By arguments similar to those of Solan and Vieille (2001, Section 2.6),
our results apply to the stronger notion of \emph{uniform equilibrium}.

General quitting games in which each player has a single continue action are called \emph{quitting games}.
Flesch, Thuijsman, and Vrieze (1997)
studied a specific three-player quitting game and identified the set of its uniform equilibrium payoffs.

Solan (1999) proved that every three-player absorbing game admits an $\ep$-equilibrium,
for every $\ep > 0$.
To date it is not known whether this result extends to absorbing games with more than three players;
for partial results on the existence of undiscounted equilibrium in multiplayer quitting games,
see Solan and Vieille (2001), Simon (2012), and Solan and Solan (2019).

In this paper we will be interested in the class of positive recursive general quitting games,
which we define now.

\begin{definition}
A general quitting game $\Gamma = (I,(A_i^c)_{i \in I},u)$ is \emph{recursive}
if $u(a) = \vec{0}$ for every nonabsorbing action profile $a \in \times_{i \in I} A_i^c$.
A general quitting game $\Gamma = (I,(A_i^c)_{i \in I},u)$ is \emph{positive} if $u_i(a) \geq 0$
for every player $i \in I$ and every action profile $a \in A$.
\end{definition}

\subsection{Sunspot Equilibrium}

We enrich the general quitting game $\Gamma$ by introducing a public correlation device:
at the beginning of every stage $t \in \dN$ the players observe a public signal $y^t \in [0,1]$ that is drawn according to the uniform distribution,
independently of past signals and play.
The extended game is denoted by $\Gamma^E$.
The set of finite histories in the game $\Gamma^E$ is $H^E := \cup_{t \in \dN} \left([0,1]^t \times A^{t-1}\right)$.

A \emph{strategy} of player~$i$ in the game $\Gamma^E$ is a sequence of measurable functions $\xi_i = (\xi_i^t)_{t \in \dN}$,
where $\xi^t_i : [0,1]^{t} \times A^{t-1} \to \Delta(A_i)$.
The interpretation of $\xi_i^t$ is that if the play was not terminated before stage $t$,
then at stage~$t$ player~$i$ plays the mixed action $\xi_i^t(y^1,a^1,y^2,a^2,\cdots,a^{t-1},y^t)$.

Every strategy profile $\xi = (\xi_i)_{i \in I}$ induces a probability distribution over the set of plays in the game with public correlation device,
with a corresponding expectation operator that is denoted by $\E_\xi$.
Denote by
\[ \gamma(\xi) := \E_\xi\left[\limsup_{T \to \infty} \frac{1}{T}\sum_{t=1}^T u(a^{\min\{t,t_*\}})\right] \]
the expected undiscounted \emph{payoff} under strategy profile $\xi$.
An $\ep$-equilibrium in the extended game $\Gamma^E$ is called a sunspot $\ep$-equilibrium of the original game $\Gamma$.

\begin{definition}
A strategy profile $\xi$ is a \emph{sunspot $\ep$-equilibrium} of $\Gamma$
if it is an $\ep$-equilibrium in the extended game $\Gamma^E$, that is, if
for every $i \in I$ and every strategy $\xi'_i$ of player~$i$ we have
\[ \gamma_i(\xi) \geq \gamma_i(\xi'_i,\xi_{-i})-\ep. \]
\end{definition}

Solan and Solan (2019) proved that every quitting game admits a sunspot $\ep$-equilibrium, for every $\ep > 0$.
Our main game theoretic result concerns the extension of this result to positive recursive general quitting games.

\begin{theorem}
\label{theorem:main}
Every positive recursive general quitting game admits a sunspot $\ep$-equilibrium, for every $\ep > 0$.
\end{theorem}

\subsection{Sunspot Equilibrium in Quitting Games}

In this section we restrict attention to a fixed quitting game $\Gamma$.
For each player $i \in I$ we denote his single continue action by $C_i$.
Denote by $\vec C = (C_i)_{i \in I}$ the action profile under which all players continue.
For every player $i \in I$ denote by $\vec C_{-i} = (C_j)_{j \neq i}$ the action profile in which all players except player~$i$ continue.

We will use the $\lambda$-discounted version of the game,
where the payoff is given by
\[ \gamma^\lambda(\sigma) :=
\E_\sigma\left[\lambda \sum_{t=1}^\infty (1-\lambda)^{t-1} u(a^{\min\{t,t_*\}}) \right], \]
and the concept of equilibrium is defined w.r.t.~the $\lambda$-discounted payoff.

Solan and Solan (2019) studied sunspot equilibrium in quitting games and proved the following result.%
\footnote{The main result of Solan and Solan (2019) involves the concept of normal players.
In a positive recursive quitting game, all players are normal,
hence the statement that appears here is equivalent to the one of Solan and Solan (2019).}
To state the result we need notations.
For every mixed action profile $x \in X$ denote by $p(x) := 1-\prod_{i \in I} x_i(A_i^c)$ the per-stage probability of absorption under $x$.
The mixed action profile $x$ is \emph{absorbing} if $p(x) > 0$,
and \emph{nonabsorbing} if $p(x)= 0$.

\begin{theorem}[Solan and Solan, 2019]
\label{theorem:21}
Let $\Gamma = (I,(\{C_i,Q_i\})_{i \in I}, u)$ be a positive recursive quitting game.
At least one of the following conditions holds.
\begin{enumerate}
\item[A.1]
For every $\ep > 0$ the game $\Gamma$ admits a sunspot $\ep$-equilibrium $\xi$ in which,
after every finite history, at most one player $i$ plays the action $Q_i$,
and the probability by which this player plays the action $Q_i$ is at most $\ep$.
Moreover,
the expected payoff to each player $i \in I$ after every finite history along which the play was not yet absorbed is at least
$u_i(Q_i,\vec C_{-i})$:
\[ \gamma_i(\xi \mid h^t) \geq u_i(Q_i,\vec C_{-i}), \ \ \ \forall i \in I, \forall h^t \in H^E. \]
\item[A.2]
There is $\eta > 0$ such that for every function $\lambda \mapsto x_{\lambda}$
that maps every $\lambda \in (0,1]$ to a $\lambda$-discounted stationary equilibrium $x_\lambda$ in $\Gamma$
such that $\lim_{\lambda \to 0} x_\lambda$ exists, we have
$\lim_{\lambda \to 0} p(x_\lambda) \geq \eta$.
\end{enumerate}
\end{theorem}

The next result, which follows from the continuity of the discounted payoff,
states that if the limit of stationary $\lambda$-discounted equilibria of a quitting game as the discount factor $\lambda$ goes to 0 is absorbing, then the limit is a stationary 0-equilibrium.

\begin{lemma}
\label{lemma:discounted}
Fix the set $I$ of players,
and for every $k \in \dN$ let $\Gamma^{[k]}=(I,(\{C_i,Q_i\})_{i \in I},u^{[k]})$ be a quitting game,
such that the sequence $(u^{[k]})_{k \in \dN}$ of payoff functions converges
to a payoff function $u$.
Let $(\lambda^{[k]})_{k \in \dN}$ be a sequence of discount factors that converges to 0.
For every $k \in \dN$ let $x^{[k]}$ be a stationary $\lambda^{[k]}$-discounted equilibrium in the quitting game $\Gamma^{[k]}$
such that the limit $x := \lim_{k \to \infty} x^{[k]}$ exists and satisfies $p(x) \in (0,\ep)$.
Then the quitting game $(I,(\{C_i,Q_i\})_{i \in I},u)$ admits an $\ep$-equilibrium for every $\ep > 0$.
If the quitting game $(I,(\{C_i,Q_i\})_{i \in I},u)$ is positive and recursive,
then $x$ is a stationary 0-equilibrium.
\end{lemma}

The $\ep$-equilibrium that exists in the quitting game $(I,(\{C_i,Q_i\})_{i \in I},u)$ according to Lemma~\ref{lemma:discounted} may be of two possible types:
\begin{itemize}
\item   If under $x$ at least two players quit with positive probability,
then $x$ is a stationary 0-equilibrium of $(I,(\{C_i,Q_i\})_{i \in I},u)$.
\item   If under $x$ exactly one player, say, Player~1, quits with positive probability,
then Player~1 may find it beneficial to continue rather than to quit when the other players follow $x_{-1}$.
To guarantee that such a deviation is not profitable in the game $(I,(\{C_i,Q_i\})_{i \in I},u)$,
the other players will punish Player~1 at his min-max level
if the game is not absorbed after sufficiently many stages have elapsed.
\end{itemize}
If the game $(I,(\{C_i,Q_i\})_{i \in I},u)$ is positive and recursive,
then in the second case, Player~1 cannot profit by not quitting,
hence $x$ is a stationary 0-equilibrium in this case as well.

\subsection{A Family of Auxiliary Games}

In this section we fix a positive recursive general quitting game $\Gamma$ and we define a family of auxiliary quitting games,
which will prove essential for our proof technique.
After defining these auxiliary games we will provide two simple relations between equilibria in these games and equilibria in the original game.

Fix a general quitting game $\Gamma$.
For each player $i \in I$
denote an element $\alpha_i \in \Delta(A_i^c)$ by $\alpha_i^1 C_i^1 + \cdots + \alpha_i^{k_i} C_i^{k_i}$,
where $k_i := |A^c_i|$ is the number of continue actions of player~$i$ and $A_i^c = \{C_i^1,\cdots,C_i^{k_i}\}$.
For every vector $\alpha = (\alpha_1,\cdots,\alpha_{|I|}) \in \times_{i \in I} \Delta(A_i^c)$
and every $q \in \dR^I$ define an auxiliary quitting game $\Gamma^{\alpha,q}$
that is based on $\Gamma$ and is defined as follows:
\begin{itemize}
\item Whenever player~$i$ continues, it is as if he plays each
continue action $C_i^k$ in $\Gamma$ with probability
$\alpha_i^k$, for $1 \leq k \leq k_i$.
\item The nonabsorbing payoff is $q$.
\end{itemize}
Formally, the stage payoff in the game $\Gamma^{\alpha,q}$, denoted $u^{\alpha,q}(\cdot)$,
is defined as follows,
where $\vec C_J = (C_j)_{j \in J}$, $\vec Q_J = (Q_j)_{j \in J}$, and $\vec \alpha_J = (\alpha_j)_{j \in J}$ for every subset $J$ of players.
\begin{eqnarray*}
u^{\alpha,q}(\vec C) &:=& q,\\
u^{\alpha,q}(\vec C_J, \vec Q_{I \setminus J}) &:=&
u(\vec\alpha_J,\vec Q_{I \setminus J}), \ \ \ \emptyset \subseteq J \subset I.
\end{eqnarray*}
Denote by $\gamma^{\alpha,q}$ the payoff function of the auxiliary game $\Gamma^{\alpha,q}$.

For every $i \in I$, a mixed action of player~$i$ in the auxiliary game $\Gamma^{\alpha,q}$ is equivalent to an element $\widehat x_i \in [0,1]$,
which is interpreted as the probability that player~$i$ quits.
The mixed action $\widehat x_i$ corresponds to a mixed action $x_i = x_i(\widehat x_i,\alpha_i) \in \Delta(A_i)$ in the game $\Gamma$ as follows:
\begin{equation}
\label{equ:x}
\begin{split}
&x_i(Q_i) := \widehat x_i,\\
&x_i(C_i^k) := (1-\widehat x_i) \alpha_i^k, \ \ \ \forall k \in \{1,2,\ldots,k_i\}.
\end{split}
\end{equation}
This correspondence between mixed actions in the auxiliary game $\Gamma^{\alpha,q}$ and mixed actions in the original game $\Gamma$ can be used to map
strategies in the game $\Gamma^{\alpha,q,E}$ into strategies in the game $\Gamma^E$ as follows.
For every finite history $h^t = (y^1,a^1,\cdots,y^t) \in H^E$ let
$\widehat h^t = (\widehat y^1,\widehat a^1,\cdots,\widehat y^t)$ be the history in the auxiliary game $\Gamma^{\alpha,q,E}$ that is defined as follows:
\begin{eqnarray*}
\widehat y^k &:=& y^k, \ \ \ 1 \leq k \leq t,\\
\widehat a^k_i &:=&\left\{
\begin{array}{lll}
a^k_i, & \ \ \ \ \ & a^k_i = Q_i,\\
C_i, & & a^k_i \in A^c_i.
\end{array}
\right.
\end{eqnarray*}
In words, we replace all continue actions in $h^t$ by the unique continue action of the player in the auxiliary game.
Given a strategy $\widehat \xi_i$ in the game $\Gamma^{\alpha,q,E}$
we define a strategy $\xi_i$ in the game $\Gamma^E$ by
\begin{equation}
\label{corr:strategy}
\xi_i(h^t) := x_i(\widehat \xi_i(\widehat h^t),\alpha_i), \ \ \ \forall i \in I, \forall h^t \in H^E.
\end{equation}
The reader can verify that for every strategy profile $\widehat \xi$ in $\Gamma^{\alpha,q,E}$,
if the strategy profile $\xi$ is defined as in Eq.~\eqref{corr:strategy}, then
\begin{equation}
\label{corr:payoff}
\gamma(\xi) = \widehat \gamma^{\alpha,q}(\widehat \xi).
\end{equation}

\begin{lemma}
\label{lemma:relation:1}
Let $\Gamma$ be a positive recursive general quitting game and let $\ep > 0$ be sufficiently small.
Suppose that there exist $\alpha = (\alpha_1,\cdots,\alpha_{|I|}) \in \times_{i \in I} \Delta(A_i^c)$ and $q \in \dR^I$ such that
the game $\Gamma^{\alpha,q}$ admits a sunspot $\ep$-equilibrium $\widehat\xi$ satisfying
(a) for every finite history $\widehat h^t$ in the game $\Gamma^{\alpha,q,E}$,
\[ \gamma_i(\widehat \xi \mid \widehat h^t) \geq u_i(Q_i,\vec C_{-i}), \ \ \ \forall i \in I, \]
and
(b) under $\widehat \xi$, at every stage at most one player quits, and he does so with probability at most $\ep$.
Then the game $\Gamma$ admits a sunspot $3\ep$-equilibrium.
\end{lemma}

\begin{proof}
The proof is standard, hence we provide only a sketch.
Let $\widehat\xi$ be a sunspot $\ep$-equilibrium in the game $\Gamma^{\alpha,q}$ that satisfies properties (a) and (b).
We will show that the strategy profile $\xi$ that is defined as in Eq.~\eqref{corr:strategy},
supplemented with statistical tests and threat of punishment,
satisfies the desired result.

Fix a player $i \in I$.
We will now check whether player~$i$ can profit by deviating from $\xi$.
Condition~(b) says that after every finite history, the play is $\ep$-close to $\vec C$,
and Eq.~\eqref{corr:payoff} implies that
$\gamma(\xi \mid h^t) = \widehat \gamma^{\alpha,q}(\widehat \xi \mid h^t)$ for every finite history $h^t \in H^E$.
Consequently, Condition~(a) implies that player~$i$ cannot gain more than $\ep$ by playing $Q_i$ after any finite history $h^t \in H^E$.
Thus, the only possible profitable deviation of player~$i$ is to change the probabilities by which he plays his continue actions.
Since the per-stage probability by which players quit is at most $\ep$,
and since player~$i$ plays the mixed action $\alpha_i$ until the game terminates,
provided $\ep$ is sufficiently small,
by conducting statistical tests players $I \setminus\{i\}$ can check whether player~$i$ plays his continue actions according to the mixed action $\alpha_i$,
and punish Player~1 at his min-max level if he is found deviating.
Since the game is positive and recursive, such a deviation cannot be profitable for player~$i$.
It follows that the strategy profile $\xi$ supplemented with statistical tests and threats of punishment is a $3\ep$-equilibrium in the game $\Gamma^E$.
\end{proof}

\begin{lemma}
\label{lemma:relation:2}
Let $\Gamma$ be a positive recursive general quitting game and let $\ep > 0$ be sufficiently small.
Suppose that there exist $\alpha \in \times_{i \in I} \Delta(A^c_i)$ and $q \in \dR^I$
such that the game $\Gamma^{\alpha,q}$ admits a stationary 0-equilibrium $\widehat x \in [0,1]^I$ such that $p(\widehat x) \in (0,\ep^2)$.
Then the game $\Gamma$ admits an $\ep$-equilibrium.
\end{lemma}

\begin{proof}
Consider the stationary strategy profile $x = (x_i)_{i \in I}$ in $\Gamma$ that is defined by
\[ x_i := x_i(\widehat x_i,\alpha_i). \]
Fix now a player $i \in I$.
Since $\widehat x$ is a 0-equilibrium in $\Gamma^{\alpha,q}$ it follows that
\begin{equation}
\label{equ:45}
\gamma_i(x) = \gamma^{\alpha,q}_i(\widehat x) \geq \gamma^{\alpha,q}_i(Q_i,\widehat x_{-i}) = \gamma_i(Q_i,x_{-i}).
\end{equation}
Moreover, if $\widehat x_i(Q_i) > 0$ then there is an equality in Eq.~\eqref{equ:45}.
It follows that player~$i$ cannot profit by changing the probability by which he plays the action $Q_i$.
Thus, the only profitable deviations of player~$i$ may be to change the frequency in which he plays his continue actions.
Since $p(\widehat x) < \ep^2$,
the probability that the game is absorbed in a stage in which player~$i$  plays a continue action is bounded by $\ep^2$.
Therefore,
as in the proof of Lemma~\ref{lemma:relation:1},
provided $\ep$ is sufficiently small,
every set of $|I|-1$ players can perform a statistical test that checks whether
the $|I|$'th player plays his continue actions with the frequency indicated by $x$,
and if not, punish him at his min-max level.

It follows that the stationary strategy profile $x$,
supplemented with statistical tests and threat of punishment,
is an $\ep$-equilibrium in $\Gamma$, provided $\ep$ is sufficiently small.
\end{proof}

\bigskip

To complete the proof of Theorem~\ref{theorem:main} we will consider from no on only positive recursive general quitting games
that do not satisfy the conditions of Lemmas~\ref{lemma:relation:1} and~\ref{lemma:relation:2}.
In Section~\ref{section:sketch} we provide a proof for the special case in which Player~1 has two continue actions
and each other player has a single continue action.
This case exhibits some important aspects of the proof of the general case,
and will help us explain the need for the new tools that we develop in the sequel.
The general case is proven in Section~\ref{section:proof}.

\section{The Proof for the Case $|A_1^c|=2$ and $|A_i^c|=1$ for Every $i \neq 1$}
\label{section:sketch}

In this section we prove Theorem~\ref{theorem:main} when Player~1 has two continue actions
while all other players have a single continue action.
In this case the set of mixed continue action profiles $\times_{i \in I} \Delta(A_i^c)$ is equivalent to the unit interval.
We will therefore describe a mixed continue action profile by a number $\alpha_1 \in [0,1]$ instead of $\alpha = (\alpha_i)_{i \in I}$,
with the interpretation that $\alpha_1$ is the probability that Player~1 assigns to the continue action $C_1^1$
(and $1-\alpha_1$ is the probability that he assigns to the action $C_1^2$).

One interesting aspect of the case $|A_1^c|=2$ and $|A_i^c|=1$ for every $i \neq 1$
is that it uses Browder's Theorem, which we present now,
instead of a more sophisticated fixed point theorem
that we will need for the general case.
The authors are not aware of another application of Browder's Theorem in dynamic games.

\begin{theorem}[Browder, 1960]
Let $X \subseteq \dR^n$ be a convex and open set,
let $K \subseteq X$ be convex and compact,
and let $F : [0,1] \times X \to K$ be a continuous function.
Let $C_F := \{ (t,x) \in [0,1] \times X \colon x = f(t,x)\}$ be the set of fixed points of $f$.
There is a connected component $T$ of $C_F$ such that $T \cap (\{0\} \times X) \neq \emptyset$
and $T \cap (\{1\} \times X) \neq \emptyset$.
\end{theorem}

By Brouwer's Fixed Point Theorem, every continuous function $F : X \to K$
has at least one fixed point.
Browder's Theorem states that when the function $F$ depends continuously on a one-dimensional parameter whose range%
\footnote{As John Levy pointed out to the authors, when $F$ is semialgebraic, Browder's Theorem extends to multi-dimensional compact and convex parameter sets.
\label{footnote:levy}}
is $[0,1]$,
the set of fixed points, as a function of the parameter, has a connected component whose projection to the set of parameters is $[0,1]$.

We will show that, for a given $\delta > 0$, the game $\Gamma$ admits a sunspot $\delta$-equilibrium.
To this end we will assume throughout that
Condition~(A.2) in Theorem~\ref{theorem:21} holds and that
the condition of Lemma~\ref{lemma:relation:2} does not hold.

Condition~(A.2) of Theorem~\ref{theorem:21} says%
\footnote{In fact, Theorem~\ref{theorem:21} implies that the condition holds for every $\alpha_1 \in [0,1]$ and every $q^* \in \dR^I$.}
that there exist $\alpha^*_1 \in \Delta(\{C_1^1,C_1^2\})$, $q^* \in \dR^I$,
and $\eta > 0$ such that for every function $\lambda \mapsto x_{\lambda}$
that maps every $\lambda \in (0,1]$ to a $\lambda$-discounted stationary equilibrium $x_\lambda$ in $\Gamma^{\alpha_1^*,q^*}$
such that $\lim_{\lambda \to 0} x_\lambda$ exists, we have
$\lim_{\lambda \to 0} p(x_\lambda) \geq \eta$.

\bigskip
\noindent\textbf{Step 1:} Applying Browder's Theorem.

For every $\lambda \in (0,1]$, every $\alpha_1 \in [0,1]$,
denote by $M_{\lambda}(\alpha_1) \subseteq [0,1]^I$ the set of $\lambda$-discounted stationary equilibria of the game $\Gamma^{\alpha_1,q^*}$:
\[ M_{\lambda}(\alpha_1) :=
\left\{ \widehat x \in [0,1]^I \colon \widehat x \hbox{ is a } \lambda\hbox{-discounted stationary equilibrium of } \Gamma^{\alpha_1,q^*}\right\}. \]
Denote by $M_\lambda \subseteq [0,1] \times [0,1]^I$ the graph of the function $M_\lambda(\cdot)$.

Browder's Theorem implies%
\footnote{To apply Browder's Theorem we need to show that the set $M_\lambda$ is the set of fixed points
of some continuous function $f : [0,1] \times [0,1]^I \to [0,1]^I$.
Such a function can be constructed
using the function devised in Nash (1950) to prove the existence of equilibrium in strategic form games,
by observing that a $\lambda$-discounted stationary equilibrium is a fixed point of the Shapley operator; see Fink (1964).
Browder's Theorem is applied to $K = [0,1]^I$ and $X = [-\ep,\ep]^I$.}
that there exists a connected component of $M_\lambda$ that intersects both $\{0\} \times [0,1]^I$ and $\{1\} \times [0,1]^I$.
Because the set $M_\lambda$ is semialgebraic, this in turn implies that there is a continuous path in $M_\lambda$ that
intersects both $\{0\} \times [0,1]^I$ and $\{1\} \times [0,1]^I$.

\bigskip
\noindent\textbf{Step 2:} Constructing a continuous path of equilibria.

Denote by $M := \limsup_{\lambda \to 0} M_\lambda \subseteq [0,1] \times [0,1]^I$ the set of all accumulation points of sequences in $M_\lambda$ as $\lambda$ goes to 0;
that is, $M$ is the set of all limits $\lim_{k \to \infty} (\alpha_1^k,\widehat x^{k}(\alpha_1^k))$,
where $\alpha_1^k \in [0,1]$, $\lambda^k \in (0,1]$, and $\widehat x^{k}(\alpha_1^k) \in M_{\lambda^k}(\alpha_1^k)$ for every $k \in \dN$,
such that $\lim_{k \to \infty} \lambda^k = 0$ and
the two limits $\lim_{k \to \infty} \alpha_1^k$
and $\lim_{k \to \infty} \widehat x^{k}(\alpha_1^k)$ exist.
The set $M$ is closed and semialgebraic.
Moreover,
since for every $\lambda > 0$ there is a continuous path in $M_\lambda$ that intersects $\{0\} \times [0,1]^I$ and $\{1\} \times [0,1]^I$,
it follows that there is a continuous path in $M$ that intersects $\{0\} \times [0,1]^I$ and $\{1\} \times [0,1]^I$.
This implies that there is a continuous function $\alpha_1 : [0,1] \to [0,1]$ satisfying $\alpha_1(0) = 0$ and $\alpha_1(1)=1$,
and a continuous function $\widehat x : [0,1] \to [0,1]^I$,
such that $\widehat x(s) \in M(\alpha_1(s))$, for every $s \in [0,1]$.
Let $s_0 \in [0,1]$ satisfy $\alpha_1(s_0) = \alpha_1^*$.
By the choice of $\alpha^*_1$ we have $p(\widehat x(s_0)) > 0$.

\bigskip
\noindent\textbf{Step 3:} There is $\eta' > 0$ such that $p(\widehat x(s)) \geq \eta'$ for every $s \in [0,1]$.

Assume to the contrary that the claim does not hold.
Since the function $\widehat x$ is continuous,
for every $\eta' > 0$ sufficiently small there is $s \in [0,1]$ such that $p(\widehat x(s)) = \eta'$.
By definition, $(\alpha_1(s),\widehat x(s))$ is the limit of a sequence $(\alpha_1^{[k]}(s),\widehat x^{[k]}(s))_{k \in \dN}$,
where $\widehat x^{[k]}(s)$ is a stationary $\lambda^{[k]}$-discounted equilibrium in the auxiliary game $\Gamma^{\alpha_1^{[k]}(s),q^*}$ for every $k \in \dN$
such that $\lim_{k \to \infty}\lambda^{[k]} = 0$.
By Lemma~\ref{lemma:discounted}, $\widehat x(s)$ is a 0-equilibrium in the auxiliary game $\Gamma^{\alpha_1(s),q^*}$,
and therefore the condition of Lemma~\ref{lemma:relation:2} holds, a contradiction.

\bigskip
We next show that for every $s \in [0,1]$, only Player~1 may have a profitable deviation from $x(s)$ in $\Gamma$.
We will then show that there is $s_0 \in [0,1]$ such that $x(s)$ is a stationary equilibrium of the original game $\Gamma$.

\bigskip
\noindent\textbf{Step 4:} For every $s \in [0,1]$, when the players follow the stationary strategy profile $x(s)$
no player $i\neq 1$ can profit by deviating from $x_i(s)$.
Moreover, Player~1 cannot profit by deviating to $Q_1$.

By Step~3 and Lemma~\ref{lemma:discounted},
the strategy profile $x(s)$ is a stationary 0-equilibrium in the auxiliary game $\Gamma^{\alpha_1(s),q^*}$.
In particular, any deviation in $\Gamma$ that is possible in the auxiliary game $\Gamma^{\alpha_1(s),q^*}$
is not profitable in the original game $\Gamma$, and the claim follows.

\bigskip
\noindent\textbf{Step 5:} The case that there is $s_0 \in [0,1]$ such that $\sum_{i \neq 1} \widehat x_i(s_0) = 0$.

Since $\sum_{i \neq 1} \widehat x_i(s_0) = 0$,
it follows that $\widehat x_1(s_0) > 0$.
Since the game is recursive and positive, and all players except Player~1 continue,
Player~1 cannot profit by deviating from
$x_1(s_0)$ in the game $\Gamma$.

\bigskip
\noindent\textbf{Step 6:} The case that $\sum_{i \neq 1} \widehat x_i(s) > 0$ for every $s \in [0,1]$.

For every absorbing mixed action profile $x \in X = \times_{i \in I}\Delta(A_i)$ denote by $u(x)$ the \emph{expected absorbing payoff} under $x$:
\[ u(x) := \frac{\sum_{a \in A} \left(\prod_{i \in I} x_i(a_i)\right) p(a)u(a)}{p(x)}. \]
Let $x_i(s) := x_i(\widehat x_i(s),\alpha_i(s))$, see Eq.~\eqref{equ:x}.
Denote by $u_1^1(s) := u_1(C^1_1,x_{-1}(s))$ (resp.~$u_1^2(s) := u_1(C^2_1,x_{-1}(s))$)
the payoff of Player~1 if he plays the stationary strategy $C^1_1$ (resp.~$C_1^2$)
while all other players follow the stationary strategy profile $x(s)$.
These quantities are well defined because $\sum_{i \neq 1} \widehat x_i(s) > 0$ for every $s \in [0,1]$.
Since the function $s \mapsto \widehat x(s)$ is continuous, the functions
$s \mapsto u_1^1(s)$ and
$s \mapsto u_1^2(s)$ are continuous as well.

By definition, for $s=1$, the strategy $x_1(1)$ assigns probability 0 to the action $C^2_1$.
Similarly, for $s=0$, the strategy $x_1(0)$ assigns probability 0 to the action $C^1_1$.
Consequently, if $u^1_1(1) \geq u_1^2(1)$,
then Player~1 cannot profit by deviating from $x_1(1)$ to $C_1^2$,
and therefore the stationary strategy profile $x(1)$ is a 0-equilibrium in $\Gamma$.
Similarly, if $u^1_1(0) \leq u_1^2(0)$,
then Player~1 cannot profit by deviating from $x_1(0)$ to $C_1^1$,
and therefore the stationary strategy profile $x(0)$ is a 0-equilibrium in $\Gamma$.

It is left to consider the case $u^1_1(1) < u_1^2(1)$ and $u^1_1(0) > u_1^2(0)$.
The continuity of the functions $u^1_1$ and $u^2_1$ implies that there is $s_0 \in (0,1)$
such that $u^1_1(s_0) = u^2_1(s_0)$.
But then both $C_1^1$ and $C_1^2$ yield the same payoff against $x_{-1}(s_0)$,
and therefore Player~1 cannot profit by deviating from $x_1(s_0)$ to either $C_1^1$ or $C_1^2$.
In particular, the stationary strategy profile $x(s_0)$ is a 0-equilibrium in $\Gamma$.

\bigskip

We now discuss the adaptation of the proof to the general case.
As above, the challenging case is when
there exist $\alpha^* \in \times_{i \in I} \Delta (A_i^c)$ and $q^*\in \dR^I$ such that the auxiliary game $\Gamma^{\alpha^*,q^*}$
satisfies Condition~(A.2) in Theorem~\ref{theorem:21}.
For every $\alpha \in \times_{i \in I} \Delta (A_i^c)$ denote by
$M_\lambda(\alpha) \subseteq \times_{i \in I} \Delta (A_i^c)$ the set of all $\lambda$-discounted equilibria of the game
$\Gamma^{\alpha,q^*}$,
by $M_\lambda \subseteq [0,1] \times \left(\times_{i \in I} \Delta (A_i^c)\right)$
the graph of the function $\alpha \mapsto M_\lambda(\alpha)$,
and by $M := \limsup_{\lambda \to 0} M_\lambda\subseteq [0,1] \times \left(\times_{i \in I} \Delta (A_i^c)\right)$
the set of accumulation points of the sets $(M_\lambda)_{\lambda > 0}$
as $\lambda$ goes to 0.

By Browder's Theorem one can prove that the set $M$ has a connected component, whose boundary,
when projected to $\times_{i \in I}\Delta(A_i^c)$,
coincides with the boundary of $\times_{i \in I}\Delta(A_i^c)$ (recall Footnote~\ref{footnote:levy}).
In the proof above,
to show that a stationary equilibrium exists we used in Step~6 the Mean Value Theorem.
In the general case we need to use a fixed point theorem applied to the set $M$.
In Section~\ref{section:topology} we will develop such a theorem.
Our proof utilizes the theory of intersection index, which requires $M$ to be a smooth manifold.
By Kohlberg and Mertens (1986), given the set of players and the sets of actions of the players of a strategic form game,
the equilibrium set is homeomorphic to the set of games, which is a Euclidean space.
This set, however, is not a smooth manifold.
In Section~\ref{section:equilibrium} we will prove that the equilibrium set can be uniformly approximated by smooth manifolds,
a property that will suffice for our purposes.

\section{Topological Foundations}
\label{section:topology}

In this section we present the results from topology that we need in the paper.
We refer to Guillemin and Pollack (2010) for the relevant background on manifolds,
including the definition of transversality, oriented manifolds, and the intersection index.
One should bear in mind that Guillemin and Pollack (2010) often consider the case of closed manifolds without boundary,
while in our case some manifolds have boundary.
Nevertheless, our assumptions will ensure that the results still hold, with the same proofs,
when the manifolds have boundary.

All manifolds in this paper are oriented.
In this section we use simplexes and products of simplexes,
which are not smooth manifolds in the usual definition, since their boundary is not a manifold.
One way to handle such manifolds is as manifolds with corners, see, e.g., Joyce (2010).
This issue will not arise in our results; the only place where we do care about the boundary being a manifold is in Theorem~\ref{theorem:6},
and there we will deal with it specifically.

Given $\ep > 0$ and a function $f : X \to Y$, where $Y$ is a metric space with metric $\rho_Y$,
the function $f_\ep : X \to Y$ is an \emph{$\ep$-perturbation} of $f$
if $\rho_Y(f(x),f_\ep(x)) < \ep$ for every $x \in X$.
The basic result in topology that we need is a variation of Browder's Theorem.

\begin{theorem}
\label{theorem:5}
Let $X$ be a compact $k$-dimensional manifold with boundary.
Let $U$ be an $n$-dimensional connected open boundaryless manifold.
Let $M \subseteq U \times X$ be an $n$-dimensional boundaryless manifold
that satisfies $M \cap (U \times \partial X) = \emptyset$.
Let $N$ be an $l$-dimensional compact manifold with boundary.
Let $y : N \times X \to U$ be a continuous function such that for every $\alpha \in N$ the function
$y(\alpha,\cdot) : X \to U$ is homotopic to a constant function.
Consider the function $\widetilde y : N \times X \to U \times X$ defined by
\[ \widetilde y(\alpha,x) = (y(\alpha,x),x). \]
Let $\pi : U \times X \to U$ be the projection and denote $d := \deg(\pi_{|M})$.

Then for every $\ep > 0$ there is an $\ep$-perturbation $\widetilde y_\ep$ of $\widetilde y$
such that
\begin{itemize}
\item[(a)] $\widetilde y_\ep$ is transversal to $M$,
and
\item[(b)] the manifold $M' := (\widetilde y_\ep)^{-1}(M) \subseteq N \times X$ satisfies that its boundary is contained in $\partial N \times X$.
Moreover, the projection $M' \to N$ has degree $d$.
\end{itemize}
\end{theorem}

To allow game theorists to properly interpret the data of Theorem~\ref{theorem:5}, we explain its relation to games.
Suppose that the set $I$ of players and the action sets of the players $(A_i)_{i \in I}$ are fixed.
The compact manifold with boundary $X$ will be the set of mixed action profiles in binary games,%
\footnote{A \emph{binary} game is a strategic-form game in which every player has two actions.}
namely, $X = [0,1]^I$.
The connected open boundaryless manifold $U$ will be the set of possible payoff functions in binary games, namely, $U = \dR^{2^{|I|} \times |I|}$.
The manifold $M \subseteq U \times X$ will be a smooth manifold that uniformly approximates the equilibrium set.
Let $N$ be some parameter space, which is a compact manifold with boundary, for example, a finite product of simplexes,
and let $y : N \to U$ be some continuous function that assigns a game to each parameter.
In the statement of Theorem~\ref{theorem:5}, the domain of the function~$y$ is not $N$ but $N \times X$, but to understand the theorem we ignore this point.
Extend $y$ to a function $\widetilde y : N \times X \to U \times X$ by setting $\widetilde y(\alpha,x) = (y(\alpha),x)$.
Theorem~\ref{theorem:5} roughly states that $\widetilde y^{-1}(M)$ is a manifold,
and that its boundary, when projected to $N$, contains the boundary of the parameter set $N$.
In other words, it roughly says that the equilibrium set restricted to games in the range of $y$
is a manifold whose boundary covers the boundary of the parameter set.

To prove Theorem~\ref{theorem:5} we will need a couple of observations,
which follow from the definition of the intersection index.
\begin{lemma}
\label{lemma:new1}
Let $X$, $Y$, and $Z$ be three manifolds with boundary,
let $f : X \to Y$ and $g : Y \to Z$ be smooth functions,
and let $M \subseteq Z$ be a boundaryless manifold (see Figure~\arabic{figurecounter}).
Assume that
\begin{itemize}
\item   $X$ is compact,
\item   $g$ is transversal to $M$,
\item   $M$ is disjoint of $g \circ f(\partial X)$, and
\item   $\dim M + \dim X = \dim Z$.
\end{itemize}
Then the intersection index of $g\circ f$ and $M$
is equal to the intersection index of $f$ and $g^{-1}(M)$.
\end{lemma}

\begin{equation*}
\xymatrix{
X \ar[r]^f & Y\ar[r]^g & Z \\
&& M\ar@{^{(}->}[u]}
\end{equation*}

\centerline{Figure \arabic{figurecounter}: The data of Lemma \ref{lemma:new1}.}
\addtocounter{figurecounter}{1}

\bigskip

\begin{proof}
For every $\ep > 0$ there is a smooth $\ep$-perturbation $f_\ep : X \to Y$ of $f$ such that
\begin{itemize}
\item
$f_\ep$ is transversal to $g^{-1}(M)$,
\item
$f_\ep$ is homotopic to $f$ through the homotopy function $H : [0,1] \times X \to Y$
with the condition that every points moves at most $\ep$ along the homotopy;
that is, $\rho_Y(H(0,x),H(t,x)) < \ep$ for every $t \in [0,1]$ and every $x \in X$,
where $\rho_Y$ is the metric on $Y$, and
\item
$H([0,1] \times \partial X) \cap g^{-1}(M) = \emptyset$.
\end{itemize}
Since $H(\{t\}\times \partial X) \cap g^{-1}(M) = \emptyset$ for every $t \in [0,1]$,
and since $X$ is compact,
it follows that the homotopy $H$ preserves the intersection index (see, for example, Guillemin and Pollack (2010, page 108)),
and therefore we may assume that $f$ is transversal to $g^{-1}(M)$.

It follows from the definitions of the intersection index and of the inverse image that the intersection index of $g \circ f$ and $M$
is the sum of the orientations of
$(g\circ f)^{-1}(M)$.
For the same reason, the intersection index of $f$ and $g^{-1}(M)$ is the sum of the orientations of $f^{-1}(g^{-1}(M))$.
Since the inverse image of a manifold is functorial, we get the desired result.
\end{proof}

\bigskip

Let $f : X \to Y$ be a smooth function between manifolds.
A point $y \in Y$ is a \emph{regular value of $f$}
if for every $x \in f^{-1}(y)$ the differential of $f$ at $x$, denoted $df_x$,
is onto the tangent bundle at $y$, denoted $T_y(Y)$.

\begin{lemma}
\label{lemma:new2}
Let $X$, $Y$, and $Z$ be three manifolds with boundary such that $X$ and $Z$ are compact with dimension $n$,
and $Y$ has dimension $n+k$.
Let $g : Y \to Z$ be smooth.
Assume that $X \subseteq Y$ and $\partial X \subseteq \partial Y = g^{-1}(\partial Z)$.
Then the degree of $g$ restricted to $X$
is equal to the intersection index of $X$ and $g^{-1}(Z)$, for some $z \in Z$ which is a regular value of $g$.
\end{lemma}

%
%

\begin{proof}
We will apply Lemma~\ref{lemma:new1} to $X$, $Y$, $Z$, $g$, $f$ that is the inclusion from $Y$ to $Z$,
and $M$ is a regular value $z$ of $g$ thought of as a 0-dimensional manifold with positive orientation.
By Sard's Lemma such a regular value exists.
By Lemma~\ref{lemma:new1} the intersection index of $g|_X$ and the point $z$, which is the degree of $g$,
is equal to the intersection index of the inclusion and $g^{-1}(z)$, as desired.
\end{proof}

\bigskip

\begin{proof}[Proof of Theorem~\ref{theorem:5}]
We will denote the set $X$ without its boundary by $X^\circ := X \setminus \partial X$.
Fix $\ep > 0$.
By the Transversality Theorem (see, e.g., Guillemin and Pollack (2010, Theorem 70))
there is an $\ep$-perturbation $\widetilde y_\ep$ of $\widetilde y$ that satisfies the following two conditions:
\begin{itemize}
\item   $\widetilde y_\ep$ is transversal to $M$, so that Part~(a) holds, and
\item   $\widetilde y_\ep$ is homotopic to $\widetilde y$ through the homotopy function $H : [0,1] \times N \times X \to X \times U$
with the condition that every points moves at most $\ep$ along the homotopy.
\end{itemize}
Since $X$ is compact, provided $\ep$ is sufficiently small,
along the homotopy $H$ we have $H([0,1] \times N \times \partial X) \cap M = \emptyset$.
This implies that the first claim in Part~(b) holds.
Indeed, since $M$ is boundaryless,
we deduce that $\partial (\widetilde y_\ep)^{-1}(M)$ is contained in
$\partial (N \times X) = (\partial N \times X) \cup (N \times \partial X)$.
Since $\image([0,1] \times N \times \partial X) \cap M = \emptyset$,
we obtain that $\partial (\widetilde y_\ep)^{-1}(M)$ is contained in $\partial N \times X$.

From now on we fix an arbitrary element $\alpha \in N$.
By the construction of $\widetilde y_\ep$,
the function $\widetilde y_\ep(\alpha,\cdot)$ is homotopic to $\widetilde y(\alpha,\cdot)$,
which is homotopic to $\hbox{const} \times Id : X \to U \times X$.
Moreover, by assumption, along the homotopy $\partial X$ is disjoint of $M$.
Since homotopy preserves the intersection index,
the intersection index of $\widetilde y_\ep(\alpha,\cdot)$ and $M$ is $d$.

We will now apply Lemma~\ref{lemma:new2} with the following data:
\begin{equation*}
\begin{array}{lll}
\hbox{Lemma }\ref{lemma:new2} & \ \ \ \ \ & \hbox{Our proof} \\
\hline
X & & M' = (\widetilde y_\ep)^{-1}(M), \\
Y & & N \times X^\circ,\\
Z & & N,\\
g : X \to Y & & \pi : N\times X^\circ \to N.
\end{array}
\end{equation*}
We verify that the conditions of the lemma hold for these data,
and therefore we will deduce that
the degree of the projection from $N \times X$ to $N$, restricted to $M'$,
is equal to the intersection index of $\{\alpha\} \times X^\circ$ and $(\widetilde y_\ep)^{-1}(M)$.

\begin{itemize}
\item   The set $M'$ is the inverse image of a closed set under a smooth function, hence it is closed.
Since $N$ and $X$ are compact, the manifold $M'$ is compact.
\item   By definition, the manifold $N$ is compact.
\item   We argue that $\dim(M') = \dim(N)$.
Indeed,
by definition, $\codim(M') = \dim(N \times X) - \dim(M')$,
and therefore, since $\widetilde y_\ep$ is transversal,
\begin{eqnarray*}
\dim(N \times X) - \dim(M') &=& \codim(M')\\
&=& \codim(M)\\
&=& \dim(U \times X) - \dim(M)\\
&=& n+k-n=k=\dim(X).
\end{eqnarray*}
It follows that $\dim(M') = \dim(N)$, as desired.
\item
We argue that $M' \subseteq N \times X^\circ$.
Indeed, $M' \subseteq N \times X$ and $M \cap (N \times \partial X) = \emptyset$.
\item
We note that $\pi^{-1}(\partial N) = \partial (N \times X^\circ) = \partial N \times X^\circ$.
\end{itemize}

Recall that $\alpha$ is an arbitrary element in $N$.
We now apply Lemma~\ref{lemma:new1} with the following data:
\begin{equation*}
\begin{array}{lll}
\hbox{Lemma }\ref{lemma:new1} & \ \ \ \ \ & \hbox{Our proof} \\
\hline
X & \ \ \ \ \ & \{\alpha\} \times X,\\
Y & & N \times X,\\
Z & & U \times X,\\
f : X \to Y & & \hbox{inclusion} : \{\alpha\} \times X \to N \times X,\\
g : Y \to Z & & \widetilde y_\ep : N \times X \to U \times X,\\
M & & M.
\end{array}
\end{equation*}

We note that since $N \times \partial X$ is disjoint of $M$, we also have that $g \circ f(\partial(\{\alpha\} \times X))$ is disjoint of $M$.
We leave to the reader the verification that the other conditions of Lemma~\ref{lemma:new1} hold.
We deduce from Lemma~\ref{lemma:new1}
that the intersection index of $\{\alpha\} \times X^\circ$ and $(\widetilde y_\ep)^{-1}(M)$
is equal to the intersection index of $\widetilde y_\ep(\alpha,\cdot)$ and $M$,
which is equal to $d$.
The result follows.
\end{proof}

\bigskip

The following result is a fixed point theorem for manifolds,
which is close to a result of Mertens (1989, page 597).

\begin{theorem}
\label{theorem:6}
Let $\Delta$ be a $d$-dimensional convex compact set
and let $M$ be a $d$-dimensional compact manifold with boundary.
Let $g : M \to \Delta$ be a continuous function,
and let $f : M \to \Delta$ be a smooth function that satisfies the following conditions:
\begin{itemize}
\item $\partial M \subseteq f^{-1}(\partial \Delta)$.
\item   The degree of $f$ is not zero: $\deg(f) \neq 0$.
\end{itemize}

Then there is $x \in M$ such that $f(x) = g(x)$.
\end{theorem}

\bigskip

\begin{proof}

\noindent\textbf{Step 1:} We can assume that $\Delta$ has a smooth boundary,
that the image of $g$ does not intersect $\partial\Delta$,
and that $f$ is transveral to $\partial\Delta$.

Assume that Theorem~\ref{theorem:6} holds whenever $\Delta$ has a smooth boundary and
$\image(g) \cap \partial \Delta = \emptyset$,
but does not hold without these restrictions.
Let $g_0 : M \to \Delta$ be a continuous function for which $\image(g_0) \cap \partial \Delta \neq \emptyset$.
Fix $\ep \in (0,1)$, and let $\Delta'_\ep \subseteq \Delta_\ep \subseteq \Delta$ be two convex compact subsets of $\Delta$ whose boundary is smooth,
whose Hausdorff distance from $\Delta$ is smaller than $\ep$,
and such that $\Delta'_\ep$ and $\partial\Delta_\ep$ are disjoint.
Denote $\delta := \dist(\Delta'_\ep,\partial\Delta_\ep) > 0$,
the Euclidean distance between $\Delta'_\ep$ and $\partial\Delta_\ep$.

For every $x \in M$ let $g_\ep(x)$ be the point in $\Delta'_\ep$ closest to $g_0(x)$.
Then the image of the function $g_\ep$ does not intersect $\partial\Delta_\ep$.
Let $f_\ep : M \to \Delta$ be an $\delta$-perturbation of $f$ that is transversal to $\partial\Delta'_\ep$ and coincides with $f$ on $\partial M$.
Provided $\delta$ is sufficiently small, $\deg(f_\ep) = \deg(f) \neq 0$.
Let $M' := f_\ep^{-1}(\Delta_\ep)$,
and apply the theorem to $\Delta_\ep$, $M'$, $g_\ep$, and $f_\ep$.
It follows that there exists $x_\ep \in M$ such that $f_\ep(x_\ep) = g_\ep(x_\ep)$.
Since the manifold $M$ is compact, the sequence $(x_\ep)_{\ep > 0}$ has an accumulation point $x \in M$ as $\ep$ goes to 0,
which, by continuity, satisfies $f(x) = g(x)$.

Since every convex compact set with smooth boundary is diffeomorphic to a ball,
we will assume from now on that $\Delta$ is the $d$-dimensional unit ball.

\bigskip
\noindent\textbf{Step 2:} We can assume that $g$ is smooth.

Suppose that the theorem holds whenever the function $g$ is smooth,
and let $g_0$ be an arbitrary continuous function.
To show that the result holds for $g_0$, we will consider its convolution with a sequence of smooth bump functions
that converge to a Dirac function.

Embed $M$ in a Euclidean space $\dR^m$, for $m$ sufficiently large.
Denote by $\rho$ the restriction of the standard metric on $\dR^m$ to $M$.
The metric and the orientation define a maximal form $\omega$ on $M$.

Let $\mu : \dR \to \dR$ be the smooth function defined by
\[ \mu(z) := \left\{
\begin{array}{lll}
\exp(-1/z^2) & \ \ \ \ \ & z \geq 0,\\
0 & & z < 0.
\end{array}
\right. \]
and let $\kappa : \dR^m \to \dR$ be the smooth function defined by
\[ \kappa(y) := \prod_{i=1}^m \mu(y_i+1)\mu(1-y_i), \ \ \ \forall y \in \dR^m. \]
The function $\kappa$ is smooth and vanishes outside a ball of radius 1 around the origin.
For every $\ep > 0$ define a function $\kappa_\ep : \dR^m \to \dR$ by $\kappa_\ep(y) := \kappa(\tfrac{y}{\ep})$, for every $y \in \dR^m$.
The function $\kappa_\ep$ is smooth and vanishes outside a ball of radius $\ep$ around the origin.
Finally, for every $\ep > 0$ define the convolution $g_\ep : M_\ep \to \dR$, whose domain is $M_\ep := \{x \in \dR^m \colon \rho(x,M) < \ep\}$,
the $\ep$-neighborhood of $M$, by
\[ g_\ep(x) := \frac{\int_M g(x) \kappa_\ep(x-p) \omega}{\int_M \kappa_\ep(x-p) \omega}. \]
The function $g_\ep|_M$ satisfies the conditions of the theorem and is smooth.
Since the result holds for smooth functions,
for every $\ep > 0$ there is a point $x_\ep \in M$ that satisfies $g_\ep(x_\ep) = f(x_\ep)$.
Since $M$ is compact the function $g$ is uniformly continuous,
and therefore the pointwise convergence of the functions $(g_\ep)_{\ep > 0}$ to $g$ is uniform.
Since $f$ is continuous as well it follows that
any accumulation point $x$ of the sequence $(x_\ep)_{\ep > 0}$ as $\ep$ goes to 0 satisfies $g(x) = f(x)$,
as desired.

%

\bigskip
\noindent\textbf{Step 3:} Deriving a contradiction.

Assume to the contrary that the theorem does not hold:
$f(x) \neq g(x)$ for every $x \in M$.
Let $h : M \to \partial \Delta$ be the function that is defined by
\[ h(x) := \frac{f(x)-g(x)}{\|f(x) - g(x)\|_2}. \]

Since there is no $x$ such that $f(x)=g(x)$, the function $h$ is well defined.
The function $h|_{\partial M}$ is homotopic to $f|_{\partial M}$, by the homotopy
\[ h_t(x) := \frac{f(x)-tg(x)}{\|f(x) - tg(x)\|_2}, \ \ \ \forall t \in [0,1]. \]
Note that the function $h_t$ is well defined, for every $t \in [0,1]$.
Indeed, for $t=1$ this was already established.
Consider now $t > 1$.
For every $x \in \partial M$ we have $f(x) \in \partial \Delta$, hence $\|f(x)\|_2 = 1$ while $\|tg(x)\|_2 < 1$.
In particular, the denominator of $h_t$ does not vanish.

It follows that the degree of the restricted function $h|_{\partial M}$ is equal to the degree of the restricted function $f|_{\partial M}$.
By the definition of the orientation of the boundary of $M$,
the degree of $f$ is equal to the degree
of $f|_{\partial M}$.
Thus, the degree of $h|_{\partial M}$ is nonzero.

Now,
the function $h|_{\partial M}$ can be extended to a continuous function $h : M \to \Delta$,
and hence by Guillemin and Pollack (2010, page 108, first proposition)
it follows that the intersection index of $h|_{\partial M}$ with a point $y \in \partial \Delta$ is 0.
By Lemma~\ref{lemma:new2} it follows that the degree of $h|_{\partial M}$ is 0, a contradiction.
\end{proof}

\section{Approximating the Equilibrium Set by a Smooth Manifold}
\label{section:equilibrium}

In this section we consider strategic-form games with a fixed set of players and fixed sets of actions for each player.
Kohlberg and Mertens (1986) showed that the equilibrium set when one varies the payoff function is homeomorphic to the set of games.
The goal of this section is to show that the equilibrium set can be uniformly approximated by a smooth manifold.

\begin{definition}
A \emph{strategic game form} is a pair $(I,A)$
where $I$ is a finite set of players and $A = \times_{i \in I} A_i$ is the Cartesian product of finite action sets for the players.
\end{definition}

A \emph{payoff function for player~$i$} for the strategic game form $(I,A)$ is a function $u_i : A \to \dR$,
and a \emph{payoff function} is a collection $u = (u_i)_{i \in I}$ of payoff functions for the players.
Consequently, the set of all payoff functions is equivalent to $\dR^{A \times I}$.
A triplet $(I,A,u)$ where $u$ is a payoff function for the strategic game form $(I,A)$ is a \emph{game}.

A \emph{strategy} for player~$i$ is a probability distribution $x_i \in\Delta(A_i)$,
and a \emph{strategy profile} is a collection $x = (x_i)_{i \in I}$ of strategies for the players.
It follows that the set of all strategy profiles, denoted $X$, is equivalent%
\footnote{When writing $\cup_{i \in I} A_i$ we implicitly assume that the action sets of the players are disjoint.}
to $\times_{i \in I} \Delta(A_i) \subset \dR^{\cup_{i \in I} A_i}$.
A payoff function $u_i$ for player~$i$ is extended to a function from $X$ to $\dR$ in a multilinear fashion.

A strategy profile $x \in X$ is a (Nash) \emph{equilibrium} of the game $(I,A,u)$ if $u_i(x) \geq u_i(a_i,x_{-i})$
for every player $i \in I$ and every action $a_i \in A_i$.
When the strategic game form is fixed,
the equilibrium set is the collection of all pairs of a payoff function and equilibrium in the game induced by this payoff function.

\begin{definition}
Let $(I,A)$ be a strategic game form.
The \emph{equilibrium set} of $(I,A)$ is the set
\[ M := \{ (u,x) \in \dR^{A\times I} \times X \colon x \hbox{ is an equilibrium of } (I,A,u)\} \subset \dR^{A \times I} \times \dR^{\cup_{i \in I}A_i}. \]
\end{definition}

As mentioned above, Kohlberg and Mertens (1986) proved that the set $M$ is homeomorphic to the set of games, namely, to $\dR^{A \times I}$.
An important concept that we will need is that of $O_n$-equilibria, which we define now.

\begin{definition}
Let $(I,A)$ be a strategic game form, let $u : A \to \dR^I$ be a payoff function, and let $n > 0$.
The strategy profile $x$ is an \emph{$O_n$-equilibrium} of the game $(I,A,u)$ if for every player $i \in I$ and every action $a_i \in A_i$,
\begin{equation}
\label{equ:60}
x_i(a_i) = \frac{\exp(nu_i(x,a_i))}{\sum_{a'_i \in A_i}\exp(nu_i(x,a'_i))}.
\end{equation}
\end{definition}

Standard continuity arguments show that a limit of $O_n$ equilibria as $n$ goes to infinity is a Nash equilibrium.
This observation is stated in the following lemma for future reference.

\begin{lemma}
\label{lemma:on}
Let $(n^{[k]})_{k=1}^\infty$ be a sequence of real numbers that go to infinity.
Let $(u^{[k]})_{k=1}^\infty$ be a sequence of positive payoff functions for the strategic game form $(I,A)$,
and let $(x^{[k]})_{k=1}^\infty$ be a sequence of strategy profiles
such that $x^{[k]}$ is an $O_{n^{[k]}}$-equilibrium in the game $(I,A,u^{[k]})$.
If the two limits $u := \lim_{k \to \infty} u^{[k]}$ and $x := \lim_{k \to \infty}x^{[k]}$ exist,
then the strategy profile $x$ is a Nash equilibrium in the game $(I,A,u)$.
\end{lemma}

\begin{proof}
Fix a player $i \in I$ and two actions $a_i,\widehat a_i \in A_i$.
We will prove that if $u_i(a_i,x_{-i}) > u_i(\widehat a_i,x_{-i})$ then $x_i(\widehat a_i) = 0$.
Since $u_i(a_i,x_{-i}) > u_i(\widehat a_i,x_{-i})$ it follows that there is $\delta > 0$ such that for every $k$ sufficiently large,
\begin{equation}
\label{equ:61}
u_i^{[k]}(a_i,x_{-i}^{[k]}) > u_i^{[k]}(\widehat a_i,x_{-i}^{[k]}) + \delta.
\end{equation}
Since $x^{[k]}$ is an $O_{n^{[k]}}$-equilibrium in the game $(I,A,u^{[k]})$,
we have by Eq.~\eqref{equ:60}
\begin{eqnarray*}
x_i^{[k]}(\widehat a_i)
&=& \frac{\exp(n^{[k]} u_i^{[k]}(x_{-i}^{[k]},\widehat a_i))}{\sum_{a'_i \in A_i}\exp(n^{[k]}u_i^{[k]}(x_{-i}^{[k]},a'_i))}\\
&<& \frac{\exp(n^{[k]} u_i^{[k]}(x_{-i}^{[k]},\widehat a_i))}{\exp(n^{[k]}u_i^{[k]}(x_{-i}^{[k]},a_i))}\\
&<& \frac{\exp(n^{[k]} u_i^{[k]}(x_{-i}^{[k]},\widehat a_i))}{\exp(n^{[k]}u_i^{[k]}(x_{-i}^{[k]},\widehat a_i)+n^{[k]}\delta)}
= \frac{1}{\exp(n^{[k]}\delta)},
\end{eqnarray*}
and the claim follows.
\end{proof}

\bigskip

For every real number $n$ denote the set of all $O_n$-equilibria by
\[ M_n := \{ (u,x) \colon x \hbox{ is an } O_n\hbox{-equilibrium in } u\}. \]
We will show that $M_n$ is a (smooth) manifold, and that as $n$ goes to infinity, the manifold $M_n$ converges uniformly to the equilibrium set $M$.

\begin{theorem}
\label{theorem:manifold:mn}
The set $M_n$ is an $(A \times I)$-dimensional manifold.
\end{theorem}

To prove Theorem~\ref{theorem:manifold:mn} we need to study a certain function that will be used in the definition of
the immersion%
\footnote{An \emph{immersion} is a differentiable function between differentiable manifolds whose derivative is everywhere injective (one-to-one).}
between $\dR^{A \times I}$ and $M_n$.
The keen reader will identify the origin of this function
and the proof of Theorem~\ref{theorem:manifold:mn} in the work of Kohlberg and Mertens (1986).

\begin{lemma}
\label{lemma:g}
For every $n > 0$ define the function $g^{(n)} : \dR^d \to \dR^d$ by
\[ g^{(n)}_i(x) = x_i + \frac{\exp(nx_i)}{\sum_{j=1}^d \exp(n x_j)}, \ \ \ \forall i\in \{1,2,\cdots,d\}. \]
The function $g^{(n)}$ is one-to-one, onto, and an immersion.
\end{lemma}

\begin{proof}

\noindent\textbf{Step 1:} The function $g$ is an immersion.

An $n \times n$ matrix $A$ is \emph{strictly diagonal dominant} if
(a) its diagonal entries are positive,
(b) its off-diagonal entries are negative,
and
(c) the sum of elements in each row is positive.
Note that every strictly diagonal dominant  matrix is invertible.

We first argue that the Jacobian matrix of $g$ is a strictly diagonal dominant matrix at all points.
Indeed, simple algebraic calculations show that for every $i \in \{1,2,\cdots,d\}$,
\begin{eqnarray}
\frac{\partial g_i}{\partial x_i}(x) &=& 1 + n\exp(nx_i)\left( \sum_{k \neq i} \exp(nx_k)\right) > 0,\\
\frac{\partial g_i}{\partial x_j}(x) &=& -\frac{n\exp(n (x_i+x_j)))}{\left(\sum_{k=1}^d \exp(nx_k)\right)^2} < 0, \ \ \ \forall j \neq i.
\end{eqnarray}
In particular, Conditions~(a) and~(b) hold for the Jacobian matrix of $g$ at every point $x$.
We also have
\[ \sum_{i=1}^d g_{i}(x) = 1 + \sum_{i=1}^d x_i, \]
and therefore
\[ \sum_{i=1}^d\frac{\partial g_i}{\partial x_j}(x) = 1 > 0, \ \ \ \forall j \in \{1,2,\ldots,d\}, \]
so that Condition~(c) holds as well, and the Jacobian matrix is strictly diagonal dominant at all points.
It follows that $g$ is an immersion.

\bigskip

\noindent\textbf{Step 2:} The function $g$ is onto.

To prove that $g$ is onto we will show that its image is both open and closed.
Since the Jacobian matrix of $g$ at every point $x$ is invertible,
by the Open Mapping Theorem the image of $g$ is an open set.
To show that the image of $g$ is closed,
note that $\|x-g(x)\|_2 \leq 1$ for every $x \in \dR^d$,
and consider a sequence $(y^{[k]})_{k \in \dN}$ of points in the image of $g$ that converges to a point $y$.
For each $k \in \dN$ let $x^{[k]} \in \dR^d$ satisfy $y^{[k]} = g(x^{[k]})$.
Since $\| x^{[k]} - y^{[k]}\|_2 \leq 1$,
and since the sequence $(y^{[k]})_{k \in \dN}$ converges,
it follows that there is a subsequence $(x^{[k_l]})_{l \in \dN}$ that converges to a limit $x$.
Since the function $g$ is continuous, $g(x) = y$, so that $y$ is in the image of $g$.

\bigskip

\noindent\textbf{Step 3:} The function $g$ is one-to-one.

We argue that any function whose Jacobian matrix is strictly diagonal dominant is one-to-one.
Indeed, let $f$ be such a function, assume w.l.o.g.~that $f(\vec 0) = \vec 0$,
and fix $x \neq \vec 0$.
We will show that $f(x) \neq \vec 0$.
We have
\[ f(x) = f(0) + \int_{t=0}^1 df_{tx} \cdot x \rmd t = \left( \int_{t=0}^1 df_{tx} \rmd t \right) \cdot x. \]
The matrix $\int_{t=0}^1 df_{tx} \rmd t$, as an integral of strictly diagonal dominant matrices,
is strictly diagonal dominant, hence invertible.
In particular, $\left( \int_{t=0}^1 df_{tx} \rmd t \right) \cdot x \neq \vec 0$.
\end{proof}

\bigskip

\begin{proof}[Proof of Theorem~\ref{theorem:manifold:mn}]
Kohlberg and Mertens (1986) provided an equivalent representation to games.
Let $u : A \to \dR^I$ be a payoff function.
For every $i \in I$ define two functions $\widetilde u_i : A \to \dR$ and $\overline u_i : A_i \to \dR$ by
\begin{eqnarray}
\overline u_i(a_i) &:=& \frac{1}{|A_{-i}|} \sum_{a_{-i} \in A_{-i}} u_i(a_i,a_{-i}),\\
\widetilde u_i(a) &:=& u_i(a) -\overline u_i(a_i).
\end{eqnarray}
We denote this representation by $u = \langle \widetilde u,\overline u \rangle$.

Fix $n > 0$ and define a function $z_n : M_n\to \dR^{\cup_{i \in I} A_i}$ by
\begin{equation*}
z_{n,i,a_i}(u,x) := u_i(a_i,x_{-i}) + \frac{\exp(nu_i(a_i,x_{-i}))}{\sum_{j \in I}\exp(nu_j(a_j,x_{-j}))}, \ \ \ \forall i\in I, a_i\in A_i.
\end{equation*}
Define now a function $\varphi_n : M_n \to \dR^{A \times I}$ by
\begin{equation}
\label{equ:62}
\varphi_n(u,x) := \langle \widetilde u, z_n(u,x) \rangle.
\end{equation}
Lemma~\ref{lemma:g} implies that the function $\varphi_n$ is one-to-one, onto, and an immersion.
The result follows.
\end{proof}

\bigskip

We now prove that the inverse of $g^{(n)}$ converges uniformly as $n$ goes to infinity,
and we provide an explicit form to the limit function, which is nothing but the homeomorphism defined by Kohlberg and Mertens (1986).

\begin{lemma}
\label{lemma:g2}
For every $n > 0$ let $h^{(n)} : \dR^d \to \dR^d$ be the inverse of $g^{(n)}$.
Let $h : \dR^d \to \dR^d$ be the function defined by
\[ h_i(y) := \min\{y_i,\alpha^*\}, \ \ \ \forall i=1,2,\cdots,d, \]
where $\alpha^* := \max\left\{ \alpha \in \dR \colon \sum_{i=1}^d(y_i-\alpha)_+ = 1\right\}$.
Then the sequence of functions $(h^{(n)})_{n > 0}$ converges uniformly to the function $h$.
\end{lemma}

\begin{proof}
Fix $\ep > 0$, and let $n > 0$ be sufficiently large so that $\ep > 1/(1+\exp(\ep n))$.
Fix $y \in \dR^d$ and define $x := h(y)$ and $x^{(n)} := h^{(n)}(y)$.
Assume w.l.o.g.~that $y_1 \leq y_2 \leq \cdots \leq y_d$.
By the definition of $g^{(n)}$ we have $x^{(n)}_1 \leq x^{(n)}_2 \leq\cdots \leq x^{(n)}_d$.
By the definition of $h$ we have $x_1 \leq x_2 \leq \cdots \leq x_d$.

Since
\[ \sum_{i=1}^d (y_i - \alpha^*)_+ = 1 = \sum_{i=1}^d (y_i-x^{(n)}_i) = \sum_{i=1}^d (y_i-x^{(n)}_i)_+, \]
and since $x^{(n)}_1 \leq x^{(n)}_2 \leq\cdots \leq x^{(n)}_d$,
it follows that $x^{(n)}_d \geq \alpha^* = x_d$.

For every $i \in \{1,2,\ldots,d\}$ denote
\[ \alpha_i := y_i - x_i \geq 0, \]
and
\[ \alpha^{(n)}_i := y_i - x^{(n)}_i \geq 0. \]
We now claim that $\alpha_i^{(n)} < \alpha_i + \ep$.
Indeed, assume to the contrary that for some $i \in \{1,2,\ldots,d\}$ we have $\alpha_i^{(n)} \geq \alpha_i + \ep$.
Then in particular
\[ x^{(n)}_i = y_i - \alpha^{(n)}_i \leq y_i - \alpha_i - \ep = x_i - \ep \leq x_d - \ep \leq x_d^{(n)}-\ep. \]
Therefore, by the definition of $g^{(n)}$,
\begin{eqnarray*}
\ep &\leq& \alpha_i^{(n)}
= \frac{\exp(nx_i^{(n)})}{\sum_{j=1}^d\exp(nx_j^{(n)})}\\
&\leq& \frac{\exp(nx_i^{(n)})}{\exp(nx_i^{(n)} + nx_d^{(n)})}\\
&=& \frac{1}{1+\exp\bigl( n(x^{(n)}_d - x^{(n)}_i)\bigr)}
\leq \frac{1}{1+\exp(\ep n)},
\end{eqnarray*}
a contradiction to the choice of $n$.
Since $\sum_{i=1}^d \alpha_i^{(n)} = 1 = \sum_{i=1}^d \alpha_i$,
we deduce that for every $i \in \{1,2,\ldots,d\}$ we have
\[ \alpha_i -d\ep < \alpha_i^{(n)} < \alpha_i + \ep, \]
which implies that $\|h^{(n)}(y) - h(y)\|_\infty \leq d\ep$, and the desired result follows.
\end{proof}

\bigskip

Kohlberg and Mertens (1986) proved that the following function $\varphi : M \to \dR^{A \times I}$ is a homeomorphism:
\[ \varphi(u,x) := \langle \widetilde u,z(u,x)\rangle, \ \ \ \forall (u,x) \in M, \]
where notations follow the proof of Theorem~\ref{theorem:manifold:mn} and
\begin{equation*}
z_{i,a_i}(u,x) := u_i(a_i,x_{-i}) + x_i(a_i), \ \ \ \forall i \in I, a_i \in A_i.
\end{equation*}
As a conclusion of Lemma~\ref{lemma:g2} we deduce that the manifolds $(M_n)_{n>0}$ converge to the equilibrium set $M$ in a strong sense.

\begin{theorem}
\label{theorem:converge}
For every $\ep > 0$ there is $N = N(\ep) > 0$ such that
for every $n \geq N$ we have
\[ \|\varphi^{-1}(y) - (\varphi_n)^{-1}(y)\|_2 \leq \ep, \ \ \ \forall y \in \dR^{A \times I}. \]
\end{theorem}

\section{Proof of the Main Result}
\label{section:proof}

In this section we prove Theorem~\ref{theorem:main}.
Fix a positive recursive general quitting game $\Gamma = (I,(A_i^c)_{i \in I},u)$ and $\ep_0 > 0$
such that the game $\Gamma$ admits no sunspot $\sqrt{\ep_0}$-equilibrium.
In particular,
the condition of Lemma~\ref{lemma:relation:2} does not hold for some $\alpha^* \in \times_{i \in I} \Delta(A^c_i)$ and $q \in \dR^I$.
Fix $\lambda \in (0,1]$, $n > 0$, and $\ep < \min\left\{ \tfrac{\ep_0}{7},\min\left\{\frac{1}{|A_i|\exp(n)}, i \in I\right\}\right\}$.

\bigskip
\noindent\textbf{Step 1:} Applying Theorem~\ref{theorem:5}.

Denote the set of mixed action profiles in an $I$-player binary game by $Z := [0,1]^I$;
this is a compact manifold.

Denote
\[ N := \times_{i \in I} \Delta(A_i^c). \]
Note that $N$ is a compact manifold with boundary.
The set $N \times Z$ is equivalent to the set of mixed action profiles $X$ in the original game $\Gamma$.
Indeed, for every pair $(\alpha,z) \in N \times Z$, where $\alpha = (\alpha_i)_{i \in I} \in N$ and $z = (z_i)_{i \in I} \in Z$,
corresponds the mixed action profile $x = x(\alpha,z) \in \times_{i \in I}\Delta(A_i^c)$
under which $z_i$ is the probability that player~$i$ chooses the action $Q_i$
and the product $(1-z_i)\alpha_i$
determines the probability that player~$i$ uses each of his continue actions.
Formally,
\begin{eqnarray}
x_i(Q_i) &:=& z_i,\\
x_i(C_i^k) &:=& (1-z_i) \alpha_i^k, \ \ \ \forall k \in \{1,2,\ldots,k_i\}.
\end{eqnarray}

Let $U := \dR^{2^{|I|} \times |I|}$ be the set of payoff functions for binary $|I|$-player games.
The set $U$ is a connected open boundaryless manifold.
Denote by $M \subseteq \dR^{2^{|I|} \times |I|} \times [0,1]^{|I|}$ the equilibrium set of binary games,
and by $M_n$ the manifold of $O_n$-equilibria of binary games.
Let $\pi : U \times Z \to U$ be the projection.
We can choose the orientation of $M_n$ and $U$ in such a way that the degree of $\pi|_{M_n}$ is $1$.

Let $y_\lambda : N \times Z \to U$ be the continuous function that is defined by
\begin{equation}
\label{equ:101}
y_\lambda(\alpha,z;a) :=
\left\{
\begin{array}{lll}
\lambda q + (1-\lambda)\gamma^\lambda(x(\alpha,z)), & \ \ \ \ \ & \hbox{if } a = \vec C,\\
u(\alpha_J,Q_{-J}), & & \hbox{if } a = (\vec C_J,\vec Q_{-J}).
\end{array}
\right.
\end{equation}
This is the payoff function of the binary strategic-form game that is derived from the game $\Gamma^{\alpha^*,q^*}$,
assuming players discount their payoffs and the continuation strategy profile is $x(\alpha,z)$.
Since $U$ is convex, the function $y_\lambda(\alpha,\cdot) : Z \to U$ is homotopic to a constant function,
for every $\alpha \in N$.

Let $y_{0} : N \times Z \to U$ be the function that is defined in Eq.~\eqref{equ:101} with $\lambda=0$.
For every fixed $\delta > 0$, on the region
\begin{equation*}
X^*_\delta := \left\{ (\alpha,z) \in N \times Z\colon \sum_{i \in I} z_i \geq \delta\right\}
\end{equation*}
the functions $(y_\lambda)_{\lambda \in (0,1]}$ converge uniformly to $y_{0}$ as $\lambda$ goes to 0.

For every $\lambda \in (0,1]$ let $\widetilde y_\lambda : N \times Z \to U \times Z$ be the function defined by
\[ \widetilde y_\lambda(\alpha,z) := (y_\lambda(\alpha,z),z), \ \ \ \forall \alpha \in N, \forall z \in Z. \]

By Theorem~\ref{theorem:5}
applied to $X=Z$, $U$, $M=M_n$, $N$, and $y=y_\lambda$,
there is an $\ep$-perturbation $\widetilde y_{\lambda,n,\ep}$ of $\widetilde y_\lambda$ that is transversal to $M_n$ and such that the set
$M_{\lambda,n,\ep} := (\widetilde y_{\lambda,n,\ep})^{-1}(M_n) \subseteq N \times Z$ is a
$\left(\sum_{i\in I}(|A_i|-1)\right)$-dimensional manifold whose boundary is contained in $\partial N \times Z$.

\bigskip
\noindent\textbf{Step 2:} Dividing the manifold $M_{\lambda,n,\ep}$ to absorbing and nonabsorbing points.

Lemma~\ref{lemma:relation:2} and the choice of $\ep_0$ imply that
for every $\alpha \in N$ the game $\Gamma^{\alpha,q^*}$ does not admit a stationary equilibrium
whose per-stage probability of absorption is $(0,\ep_0)$.
Since the sequence of functions $(\widetilde y_\lambda)_{\lambda > 0}$ converges uniformly to $\widetilde y_{0}$ on the region $X^*_{\ep_0}$
in which the probability of absorption is at least ${\ep_0}$,
there is $\lambda_0({\ep_0}) > 0$ such that the intersection of the image of $\widetilde y_{\lambda}$ and $M$
is disjoint of $U \times \bigl(B(\vec C,\tfrac{6{\ep_0}}{7}) \setminus B(\vec C,\tfrac{{\ep_0}}{7})\bigr)$, for every $\lambda \in (0,\lambda_0({\ep_0}))$.

By Theorem~\ref{theorem:converge},
the manifolds $(M_n)_{n \in \dN}$ converge uniformly to $M$ on every compact set of games,
hence there exists $n_0({\ep_0}) \in \dN$ such that the intersection of the image of
$\widetilde y_{0}$ and $M_n$ is disjoint of $U \times \bigl(B(\vec C,\tfrac{5{\ep_0}}{7}) \setminus B(\vec C,\tfrac{2{\ep_0}}{7})\bigr)$,
for every $n \geq n_0({\ep_0})$ and every $\lambda \in (0,\lambda_0({\ep_0}))$.

Since the function $\widetilde y_{\lambda,n,\ep}$ is an $\ep$-perturbation of $\widetilde y_\lambda$,
it follows that the intersection of the image of $\widetilde y_{\lambda,n,\ep}$ and $M_n$
is disjoint of $U \times \bigl(B(\vec C,\tfrac{4{\ep_0}}{7}) \setminus B(\vec C,\tfrac{3{\ep_0}}{7})\bigr)$,
for every $n \geq n_0({\ep_0})$, every $\lambda \in (0,\lambda_0({\ep_0}))$, and every $\ep \leq \tfrac{{\ep_0}}{7}$.

We can therefore divide $M_{\lambda,n,\ep}$ into two disjoint parts: the points in $N \times B(\vec C,\tfrac{3{\ep_0}}{7})$
and the points in $N \times \bigl(Z \setminus B(\vec C,\tfrac{4{\ep_0}}{7})\bigr)$,
denoted respectively $M^{in}_{\lambda,n,\ep}$ and $M^{out}_{\lambda,n,\ep}$.
In particular, for every $(\alpha,z) \in M^{out}_{\lambda,n,\ep}$ we have $\sum_{i\in I} z_i \geq \tfrac{4{\ep_0}}{7}$.

By the choice of $q^*$, the intersection $\widetilde y_1(\{\alpha^*\} \times Z) \cap M_n$ is disjoint of $B(\vec C,\tfrac{4{\ep_0}}{7})$,
for every $n \geq n_0(\ep)$.
By Theorem~\ref{theorem:5}, the projection $f : N \times Z \to N$, restricted to the set $M_{\lambda,n,\ep}$,
has degree 1.
When restricted to $M^{in}_{\lambda,n,\ep}$, the projection $f$ is not onto and therefore it has degree 0.
It follows that the projection $f|_{M^{out}_{\lambda,n,\ep}}$ has degree 1.

\bigskip
\noindent\textbf{Step 3:} Applying Theorem~\ref{theorem:6}.

In an $O_n$-equilibrium of a binary game whose payoffs are in the interval $[0,1]$,
each action of player~$i$ is played with probability at least $\frac{1}{|A_i|\exp(n)}$.
It follows that for every $(\alpha,z) \in M^{out}_{\lambda,n,\ep}$ we have $z_i^{\lambda,n,\ep} \geq \frac{1}{|A_i|\exp(n)}-\ep$,
which is positive by the choice of $\ep$.
For every $n > 0$ and every $\ep \geq 0$ denote
\[ Z_{n,\ep} := \left\{ z \in Z \colon \sum_{i \in I} z_i \geq \frac{4\ep_0}{7}, z_i \geq \frac{1}{|A_i|\exp(n)}-\ep, \ \ \ \forall i \in I\right\}. \]
The set $Z_{n,\ep}$ is nonempty and compact, and, as mentioned above, it satisfies $M^{out}_{\lambda,n,\ep} \subseteq N \times Z_{n,\ep}$.
For every $z \in Z_{n,\ep}$ and every $i \in I$ we have $z_i > 0$,
hence the absorbing payoff $u_i(C^k_i,x_{-i}(\alpha,z))$ is well defined for every $k \in \{1,2,\ldots,k_i\}$.

Define a continuous function $g^{[n]} : N \times Z_{n,\ep} \to N$ by
\[ g^{[n]}_{i}(\alpha,z) := \sum_{k=1}^{k_i} \frac{\exp(nu_i(C^k_i,x_{-i}(\alpha,z)))}{\sum_{l=1}^{k_i}\exp(nu_i(C^l_i,x_{-i}(\alpha,z)))} C_i^k; \]
that is, $g^{[n]}_{i}(\alpha,z)$ is the probability distribution that assigns probability
$\frac{\exp(nu_i(C^k_i,x_{-i}(\alpha,z)))}{\sum_{l=1}^{k_i}\exp(nu_i(C^l_i,x_{-i}(\alpha,z)))}$
to the continue action $C_i^k$.

We would like to apply Theorem~\ref{theorem:6} with $\Delta = N$, $M = M^{out}_{\lambda,n,\ep}$, $f : N \times Z \to Z$ the natural projection, and $g = g^{[n]}$.
We need to verify that $\partial M^{out}_{\lambda,n,\ep} \subseteq f^{-1}(\partial N)$.
Since $\widetilde y_{\lambda,n,\ep}$ is transversal to $M_n$,
it follows that
\[ \partial M^{out}_{\lambda,n,\ep} \subseteq \partial M_{\lambda,n,\ep} \subseteq \partial (\dom(\widetilde y^{\lambda,\ep})) =
\partial(N \times Z) = (\partial N \times Z) \cup (N \times \partial Z). \]
Since every $O_n$-equilibrium is completely mixed,
$\partial M^{out}_{\lambda,n,\ep}$ is disjoint of $N \times \partial Z$, and hence
$\partial M^{out}_{\lambda,n,\ep} \subseteq \partial N \times Z$, so that indeed
$\partial M^{out}_{\lambda,n,\ep} \subseteq f^{-1}(\partial N)$.
By Theorem~\ref{theorem:6}
we obtain the existence of a point $(\alpha_{\lambda,n,\ep},z_{\lambda,n,\ep}) \in M^{out}_{\lambda,n,\ep}$ that satisfies
\[ g^{[n]}(\alpha_{\lambda,n,\ep},z_{\lambda,n,\ep}) = \alpha_{\lambda,n,\ep}. \]
The fact that $(\alpha_{\lambda,n,\ep},z_{\lambda,n,\ep}) \in M^{out}_{\lambda,n,\ep}$ has two implications:
\begin{itemize}
\item   Under the strategy profile $x(\alpha_{\lambda,n,\ep},z_{\lambda,n,\ep})$ the per-stage probability of absorption is bounded away from 0:
$\sum_{i \in I} z_{\lambda,n,\ep,i} \geq \tfrac{4\ep_0}{7}$.
\item   $\widetilde y_{\lambda,n,\ep}(\alpha_{\lambda,n,\ep},z_{\lambda,n,\ep}) \in M_n$.
\end{itemize}

\bigskip
\noindent\textbf{Step 4:} Taking limits.

We let $\ep$ go to 0, then $\lambda$ go to 0, and finally $n$ go to infinity.
Since the set $N\times Z$ is compact,
for every fixed $n > 0$ and $\lambda \in (0,1]$,
the sequence $(\alpha_{\lambda,n,\ep},z_{\lambda,n,\ep})_{\ep > 0}$ has an accumulation point
$(\alpha_{\lambda,n},z_{\lambda,n}) \in N \times Z_{n,0}$ as $\ep$ goes to 0.
By continuity this accumulation points satisfies the following properties:
\begin{eqnarray*}
&& g^{[n]}(\alpha_{\lambda,n},z_{\lambda,n}) = \alpha_{\lambda,n},\\
&& \sum_{i \in I} z_{\lambda,n,i} \geq \tfrac{4\ep_0}{7},\\
&& \widetilde y_{\lambda}(\alpha_{\lambda,n},z_{\lambda,n}) \in M_n.
\end{eqnarray*}

For every fixed $n > 0$
consider an accumulation point of the sequence $x(\alpha_{\lambda,n},z_{\lambda,n})_{\lambda \in (0,1]}$ as $\lambda$ goes to 0,
denoted $(\alpha_{n},z_{n})$.
Since $y_\lambda$ converges uniformly to $y_{0}$ on $Z_{n,\ep_0}$,
we deduce that
\begin{itemize}
\item[(D.1)]   $g^{[n]}(\alpha_{n},z_{n}) = \alpha_{n}$.
\item[(D.2)]   The strategy profile $z_n$ is absorbing: $\sum_{i \in I} z_{n,i} \geq \tfrac{4\ep_0}{7}$.
\item[(D.3)]   The strategy profile $z_n$ is an $O_n$-equilibrium in the binary game $y_{0}(\alpha_{n},z_{n})$.
\end{itemize}

Consider now an accumulation point $(\alpha,z)$ of the sequence $(\alpha_{n},z_{n})_{n \in \dN}$ as $n$ goes to infinity.
We will show that the strategy profile $x(\alpha,z)$ is a stationary equilibrium in the game $\Gamma$.

By continuity $\sum_{i \in I} z_i \geq \tfrac{4\ep_0}{7}$,
and therefore this strategy profile is absorbing.
By Lemma~\ref{lemma:on},
the stationary strategy profile $z$ is a 0-equilibrium in the binary game $y_{0}(\alpha,z)$.

Fix a player $i\in I$ such that whatever he plays, the play is absorbed;
that is, $\sum_{j \neq i} z_j > 0$.
We will show that player~$i$ is indifferent among all actions in the support of $\alpha_i$.
Indeed, fix two continue actions $a_i,a'_i \in A_i^c$ of player~$i$.
If $u_i(a_i,x_{-i}(\alpha,z)) < u_i(a'_i,x_{-i}(\alpha,z))$,
then there is $\eta > 0$ such that
$u_i(a_i,x_{-i}(\alpha,z)) < u_i(a'_i,x_{-i}(\alpha,z)) - \eta$.
Consequently, for every $n$ sufficiently large we have
\[ u_i(a_i,x_{-i}(\alpha_n,z_n)) < u_i(a'_i,x_{-i}(\alpha_n,z_n)) - \eta. \]
By the definition of $g^{[n]}$ and by (D.1), this implies
\[ \lim_{n \to \infty} \frac{\alpha_{n,i}(a_i)}{\alpha_{n,i}(a'_i)}
= \lim_{n \to \infty} \frac{\exp(nu_i(a_i,x_{-i}(\alpha_n,z_n)))}{\exp(nu_i(a'_i,x_{-i}(\alpha_n,z_n)))} = 0. \]
In particular, under the mixed action $x_i(\alpha,z)$ the action $a_i$ is selected with probability 0.

Since $z_n$ is an $O_n$-equilibrium of the binary game,
if $z_i > 0$, then player~$i$ is indifferent between continuing and quitting.

It is left to consider the case that player~$i$ is the sole player who quits with positive probability: $\sum_{j \neq i} z_j = 0$.
It is standard to show that for every $\delta > 0$
the stationary strategy profile $x(\alpha_n,z_n)$,
supplemented with statistical tests, is a $\delta$-equilibrium, provided $n$ is sufficiently large.

\section{Extensions}
\label{section:extensions}

In this paper we proved the existence of a sunspot $\ep$-equilibrium in the class of positive recursive general quitting games.
A natural question is whether our techniques can be applied to more general classes of games.
These include
(a) general quitting games that are not necessarily recursive and positive,
that is, the nonabsorbing payoff may depend on the continue actions that the players play;
(b) general quitting games in which players have more than one quitting action,
as well as more than one continue action;
(c) games in which the absorption structure is not rectangular; and
(d) games with more than one nonabsorbing state.

Regarding extension~(b), our proof can be adapted to this case when the game is recursive and positive,
see Munk and Solan (2019).
Regarding extension~(c), some results in this direction are provided in Solan and Solan (2018) and Munk and Solan (2019).
We hope that future research will shed more light on extensions~(a) and~(d).

\end{document}